\documentclass[11pt]{amsart}
\usepackage{amssymb,color}
\usepackage{tikz}
\theoremstyle{plain}
\newtheorem{theorem}{Theorem}[section]

\newtheorem{remark}{Remark}[section]

\newtheorem{proposition}{Proposition}[section]

\theoremstyle{definition}
\newtheorem{definition}{Definition}
\theoremstyle{remark}

\allowdisplaybreaks

\long\def\symbolfootnote[#1]#2{\begingroup
	\def\thefootnote{\fnsymbol{footnote}}\footnote[#1]{#2}\endgroup}

\begin{document}
	\title[Anisotropic $(p,q)$-equations]
	{Anisotropic $(p,q)$-equations with convex and negative concave terms}
	\author[N.S. Papageorgiou, D.D. Repov\v{s} and C. Vetro]{Nikolaos S. Papageorgiou, Du\v{s}an D. Repov\v{s} and Calogero Vetro}
	\address[Nikolaos S. Papageorgiou]{Department of Mathematics, National Technical University, Zografou campus, 15780, Athens, Greece}
	\email{npapg@math.ntua.gr}
	\address[Du\v{s}an D. Repov\v{s}]{Faculty of Education and Faculty of Mathematics and Physics,
		University of Ljubljana \& Institute of Mathematics, Physics and Mechanics,
		SI-1000, Ljubljana, Slovenia}
	\email{dusan.repovs@guest.arnes.si}
	\address[Calogero Vetro]{Department of Mathematics and Computer Science, University of Palermo, Via Archirafi 34, 90123, Palermo, Italy}
	\email{calogero.vetro@unipa.it}
	
	
	\date{}
	\maketitle

\begin{abstract}
	We consider a parametric Dirichlet problem driven by the anisotropic $(p,q)$-Laplacian and with a reaction which exhibits the combined effects of a superlinear (convex) term and of a negative sublinear term. Using variational tools and critical groups we show that for all small values of the parameter, the problem has at least three nontrivial smooth solutions, two of which are of constant sign (positive and negative). 
\end{abstract}

\section{Introduction}\label{sec:1}

Let $\Omega \subseteq \mathbb{R}^N$ be a bounded domain with a $C^2$-boundary $\partial \Omega$. In this paper we study the following parametric anisotropic Dirichlet problem 
\begin{equation}\label{eq0}\tag{$P_\lambda$} 
\begin{cases}-\Delta_{p(z)} u(z)-\Delta_q u(z)= f(z,u(z))-\lambda |u(z)|^{\tau(z)-2}u(z)  & \mbox{in } \Omega,\\ u \Big|_{\partial \Omega} =0, \, 1< \tau(z)<q<p(z)<N \mbox{ for all }z \in \overline{\Omega}, \, \lambda>0.& \end{cases}
\end{equation}
Given $r \in  C^{0,1}(\overline{\Omega})$ (= the space of Lipschitz continuous functions on $\overline{\Omega}$) with $1< r_-=\min\limits_{\overline{\Omega}}r$, by    $\Delta_{r(z)}$ we denote the anisotropic $r$-Laplacian defined by
$$\Delta_{r(z)} u= \mbox{div }(|\nabla u|^{r(z)-2}\nabla u) \quad \mbox{ for all $u \in W_0^{1,r(z)}(\Omega)$ {\color{black}(see Section \ref{sec:2})}.}$$

If $r(\cdot)$ is constant, then we have the standard $r$-Laplacian denoted by $\Delta_r$. In problem \eqref{eq0} above, we have the sum of two such operators, one with variable exponent and the other with constant exponent. In the reaction (the right hand side of \eqref{eq0}), we have the combined effects of two distinct nonlinearities.

  One is the Carath\'eodory function $f(z,x)$ (that is, for all $x \in \mathbb{R}$, $z \to f(z,x)$ is measurable and for a.a. $z \in \Omega$, $x \to f(z,x)$ is continuous). We assume that $f(z,\cdot)$ is $(p_+-1)$-superlinear  ($p_+=\max\limits_{\overline{\Omega}}p$) but it {\color{black} needs} not satisfy the (common in such cases) Ambrosetti-Rabinowitz condition, see also Papageorgiou-R\u{a}dulescu-Repov\v{s} \cite{Ref12a} (Robin problem).  This term represents a ``convex'' contribution to the reaction. 

The other nonlinearity is the parametric function $x \to -\lambda |x|^{\tau(z)-2}x$ with $\tau \in C(\overline{\Omega})$ such that $1< \tau(z)<q$ for all $z \in \overline{\Omega}$. Therefore this term is $(q-1)$-sublinear (``concave'' term). Thus the reaction of \eqref{eq0} corresponds to a ``concave-convex'' problem, but with an essential difference. The concave (sublinear) term enters in the equation with a negative sign and this changes the geometry of the problem.

 In the past, problems with a negative concave term were studied by Perera \cite{Ref19}, de Paiva-Massa \cite{Ref9}, Papageorgiou-R\u{a}dulescu-Repov\v{s} \cite{Ref10} (Robin problems) for semilinear equations driven by the Laplacian, and by Papageorgiou-Winkert \cite{Ref16} for resonant $(p,2)$-equations. All the aforementioned works deal with isotropic equations and the perturbation $f(z,\cdot)$ is $(p-1)$-linear.

 Using variational tools from the critical point theory and critical groups {\color{black}(see Section \ref{sec:2})}, we show that for all sufficiently small $\lambda>0$, problem \eqref{eq0} has at least three nontrivial smooth solutions. Two of these solutions have constant sign (one is positive and the other negative). It is an interesting open question, whether this multiplicity theorem {\color{black}still holds} when the exponent $q$ is also variable and whether we can show that the third solution is nodal (sign-changing). 
 
 {\color{black} For the hypotheses $H_0$ and $H_1$ involved in our theorem, we refer to Section \ref{sec:2}. Also $C_+=\{u \in C_0^1(\overline{\Omega}): u(z)\geq 0 \mbox{ for all } z \in \overline{\Omega}\}$.  }
 
	\begin{theorem}	\label{th9} If hypotheses $H_0$ and $H_1$ hold, then for all sufficiently small $\lambda >0$, problem \eqref{eq0} has at least three nontrivial solutions $u_0 \in C_+ \setminus \{0\}$, $v_0 \in (-C_+) \setminus \{0\}$, and $y_0 \in C_0^1(\overline{\Omega})\setminus \{0\}$.
\end{theorem}

To have a more complete picture of the relevant literature, we mention that the standard isotropic concave-convex problems (the concave term having a positive sign), were first considered by Ambrosetti-Brezis-Cerami \cite{Ref1} for semilinear equations driven by Dirichlet Laplacian. Their work was extended to nonlinear equations driven by the $p$-Laplacian by Garcia Azorero-Peral Alonso-Manfredi \cite{Ref5}. Since then appeared several works with further generalizations.  {\color{black} Just to quote a few} we mention the works of 
 Gasi\'nski-Papageorgiou \cite{Ref7,Ref8}, Papageorgiou-Repov\v{s}-Vetro \cite{Ref13a,Ref13b}, Papageorgiou-Vetro-Vetro \cite{Ref14,Ref15}, Papageorgiou-Winkert \cite{Ref17}, and the recent papers of Papageorgiou-Qin-R\u{a}dulescu  \cite{Ref9a} and Papageorgiou-R\u{a}dulescu-Repov\v{s} \cite{Ref12} on anisotropic equations. {\color{black} In all these works the concave term enters in the equation with a positive sign and this permits the use of the strong maximum principle which provides more structural information concerning the solution. This extra information allows us to use the result relating Sobolev and H\"{o}lder minimizers. In the present setting this is no longer possible and the geometry changes requiring a new approach.}

\section{Preliminaries}\label{sec:2}

The analysis of problem \eqref{eq0} uses variable Lebesgue and Sobolev spaces. A detailed presentation of these spaces can be found in the books of Cruz Uribe-Fiorenza \cite{Ref2} and of Diening-Hajulehto-H\"ast\"o-R\.{u}\v{z}i\v{c}ka \cite{Ref3}. 

Let $E_1=\{r \in C(\overline{\Omega}) \, : \, 1<r_- {\color{black}=\min\limits_{\overline{\Omega}} r \}}$.  {\color{black} In general, for} any $r \in E_1$, we set 
$$r_-=\min\limits_{\overline{\Omega}}r \mbox{ and } r_+=\max\limits_{\overline{\Omega}}r .$$

Also let $M(\Omega)=\{u : \Omega \to \mathbb{R} \mbox{ measurable}\}$. We identify two such functions which differ only on a Lebesgue null set. Given $r \in E_1$, we define the variable Lebesgue space $L^{r(z)}(\Omega)$ by
$$ L^{r(z)}(\Omega)=\left\{ u \in M(\Omega) :  \int_\Omega |u(z)|^{r(z)}dz <  {\color{black}+ \infty }  \right\}.  $$

This  space is equipped with the so-called ``Luxemburg norm'', defined by
$$\|u \|_{r(z)}= \inf \left\{\lambda >0  :   \int_\Omega\left(\frac{|u(z)|}{\lambda}\right)^{r(z)}dz\leq 1  \right\}.$$

 {\color{black} The space $L^{r(z)}(\Omega)$ endowed with this norm} becomes a Banach space which is separable and uniformly convex (hence reflexive)  {\color{black} (see \cite{Ref3}, p. 67)}.  For $r \in E_1$ by $r^\prime(\cdot)$
 we denote the variable conjugate exponent to $r(\cdot)$, that is, $\frac{1}{r(z)}+\frac{1}{r^\prime (z)}=1$ for all $z \in \overline{\Omega}$. Evidently,  $r^\prime \in E_1$ and
 $$ {\color{black}(L^{r(z)}(\Omega))^\ast}=L^{r^\prime(z)}(\Omega).$$
 
 Moreover, we have a  H\"older-type inequality, namely
$$\int_\Omega |u(z)v(z)|dz \leq \left[\frac{1}{r_-}+\frac{1}{r^\prime_-}\right]\|u\|_{r(z)}\|v\|_{r^\prime(z)} \mbox{ for all $u \in L^{r(z)}(\Omega)$, $v \in L^{r^\prime(z)}(\Omega)$ }$$
{\color{black} (see \cite{Ref2}, p. 27)}.
In addition, if $r, \widehat{r} \in E_1$ and $r(z) \leq \widehat{r}(z)$ for all $z \in \overline{\Omega}$, then $L^{\widehat{r}(z)}(\Omega) \hookrightarrow L^{r(z)}(\Omega)$ continuously {\color{black} (see \cite{Ref2}, pp. 37-38)}.

Using the variable Lebesgue spaces, we can define the corresponding variable  Sobolev spaces. {\color{black} Taken} $r \in E_1$, then 
$$ W^{1,r(z)}(\Omega)=\left\{ u \in L^{r(z)}(\Omega) :  |\nabla u| \in L^{r(z)}(\Omega)    \right\},  $$
{\color{black} where $\nabla u$ denotes} the weak gradient of $u$. This space is equipped with the norm 
$$\|u\|_{1,r(z)}=\|u\|_{r(z)}+ \|\nabla u\|_{r(z)}\quad \mbox{for all $u \in W^{1,r(z)}(\Omega)$,}$$
with $\|\nabla u\|_{r(z)}=\| |\nabla u| \|_{r(z)}$. If $r \in E_1 \cap C^{0,1}(\overline{\Omega})$, then we define also 
$$ W_0^{1,r(z)}(\Omega)=\overline{C_c^\infty(\Omega)}^{\|\cdot\|_{1,r(z)}}.$$

Both $ W^{1,r(z)}(\Omega)$ and $ W_0^{1,r(z)}(\Omega)$ are Banach spaces which are separable and uniformly convex (thus reflexive) {\color{black} (see \cite{Ref3}, p. 245)}. The critical Sobolev exponent $r^\ast(\cdot)$ is defined by 
$$r^\ast(z)=\begin{cases} \dfrac{Nr(z)}{N-r(z)} & \mbox{if } r(z) <N,\\ +\infty & \mbox{if }N \leq r(z).\end{cases}$$

For $r,p \in C(\overline{\Omega})$ with $1<r_-,p_+<N$ and $1 \leq p(z) \leq r^\ast(z)$ for all $z \in \overline{\Omega}$ (resp. $1 \leq p(z) < r^\ast(z)$ for all $z \in \overline{\Omega}$), then we have
$$ W^{1,r(z)}(\Omega) \hookrightarrow L^{p(z)}(\Omega) \mbox{ continuously}$$
$$ (\mbox{resp. }W^{1,r(z)}(\Omega) \hookrightarrow L^{p(z)}(\Omega) \mbox{ compactly}),$$
{\color{black} (see \cite{Ref3}, p. 259)}.
The same embeddings are also valid for $W_0^{1,r(z)}(\Omega)$. We mention that on $W_0^{1,r(z)}(\Omega)$ ($r \in  C^{0,1}(\overline{\Omega})$), the  Poincar\'e inequality holds. {\color{black} Recall that the Poincar\'e inequality says that there exists $c=c(\Omega)>0$ such that $\|u\|_{r(z)}\leq c\| \nabla u\|_{r(z)}$ for all $u \in W_0^{1,r(z)}(\Omega)$ (see \cite{Ref3}, p. 249).} So, on
 $W_0^{1,r(z)}(\Omega)$ we can use the following norm 
$$\|u\|=  \|\nabla u\|_{r(z)}\quad \mbox{for all $u \in W_0^{1,r(z)}(\Omega)$.}$$

In what follows, we shall denote
by $\rho_r(\cdot)$  the modular function 
$$\rho_r(u)=\int_\Omega |u(z)|^{r(z)}dz \quad \mbox{for all $u \in L^{r(z)}(\Omega)$.}$$

If $u \in W^{1,r(z)}(\Omega)$ or $u \in W_0^{1,r(z)}(\Omega)$, then $\rho_r(\nabla u)=\rho_r(|\nabla u|)$. The norm $\|\cdot \|_{r(z)}$ and the modular function $\rho_r(\cdot)$  are closely related {\color{black} (see \cite{RefFD}, Proposition 2.1) }.

\begin{proposition}
	\label{prop1} If $r \in E_1$ and $u  \in L^{r(z)}(\Omega) \setminus \{0\}$, then the following statements hold:
\begin{itemize}
	\item[(a)]	$\|u\|_{r(z)}=\theta \Leftrightarrow \rho_r\left(\frac{u}{\theta}\right)=1$ for all $\theta>0$;\\
		\item[(b)]	$\|u\|_{r(z)}< 1 \mbox{ (resp. $=1$, $>1$)} \Leftrightarrow \rho_r(u)< 1 \mbox{ (resp. $=1$, $>1$)}$;\\
		\item[(c)] $\|u\|_{r(z)}<1 \Rightarrow  \|u\|_{r(z)}^{r_+}\leq \rho_r(u) \leq  \|u\|_{r(z)}^{r_-}$;\\ 	\item[(d)] $\|u\|_{r(z)}>1 \Rightarrow  \|u\|_{r(z)}^{r_-}\leq \rho_r(u) \leq  \|u\|_{r(z)}^{r_+}$;\\
			\item[(e)]  $\|u\|_{r(z)}\to 0  \mbox{ (resp. $\|u\|_{r(z)}\to {\color{black} +\infty}$)}\Leftrightarrow \rho_r(u)\to 0 \mbox{ (resp. $\rho_r(u)\to {\color{black}+\infty}$)}$.
\end{itemize}	
\end{proposition}

We know that for $r \in E_1 \cap C^{0,1}(\overline{\Omega})$, we have
$$W_0^{1,r(z)}(\Omega)^\ast=W^{-1,r^\prime(z)}(\Omega) \quad \mbox{{\color{black} (see \cite{Ref3}, pp. 378-379)}}.$$

Consider the operator $A_{r(z)}: W_0^{1,r(z)}(\Omega) \to  W^{-1,r^\prime(z)}(\Omega)$  defined by
\begin{equation}\label{eq-Az}\langle A_{r(z)}(u),h \rangle = \int_\Omega |\nabla u(z)|^{r(z)-2}(\nabla u, \nabla h)_{\mathbb{R}^N}dz \quad \mbox{for all } u,h \in W_0^{1,r(z)}(\Omega),\end{equation}
 {\color{black} where $(\cdot, \cdot)_{\mathbb{R}^N}$ is the inner product in $\mathbb{R}^N$. This operator has the following properties (see \cite{RefFZ}, Proposition 2.9).}

\begin{proposition}\label{prop2}
If $r \in E_1 \cap C^{0,1}(\overline{\Omega})$, then the operator $A_{r(z)}(\cdot)$ is  bounded $($that is, it maps bounded sets to bounded sets$)$, continuous, strictly monotone $($thus also maximal monotone$)$ and of type $(S)_+$  $($that is, $u_n  \xrightarrow{w} u $ in $W_0^{1,r(z)}(\Omega)$ and $\limsup\limits_{n \to \infty} \langle A_{r(z)}(u_n), u_n - u \rangle \leq 0$ imply that $u_n \to u$ in $W_0^{1,r(z)}(\Omega)).$ 
\end{proposition}

Let $X$ be a Banach space, $\varphi \in C^1(X,\mathbb{R})$  and $c \in \mathbb{R}$. We set
\begin{align*}& K_\varphi  = \{u \in X : \varphi'(u) =0 \} \quad \mbox{(the critical set of $\varphi$),}\\
&	\varphi^c  = \{u \in X : \varphi(u) \leq c \}.\end{align*}

Let $(Y_1,Y_2)$ be a topological pair such that  $Y_2 \subseteq Y_1 \subseteq X$ and $k \in \mathbb{N}_0$. By $H_k(Y_1,Y_2)$ we denote the $k^{th}$-relative singular homology group with integer coefficients. If $u \in K_\varphi$ is isolated and $c=\varphi(u)$, then the critical groups of $\varphi$ at $u$ are defined by
$$C_k(\varphi,u)=H_k(\varphi^c \cap U, \varphi^c \cap U \setminus \{u\}) \quad \mbox{for all } k \in \mathbb{N}_0,$$
with $U$ a neighborhood of $u$ such that $K_\varphi \cap \varphi^c \cap U  = \{u\}$ {\color{black} (see \cite{Ref11}, Chapter 6)}. The excision property of singular homology, implies that the above definition is independent of the choice of the  isolating neighborhood $U$. For details we refer to Papageorgiou-R\u{a}dulescu-Repov\v{s} \cite{Ref11}, {\color{black} Chapter 6, where the reader can find explicit computations of the critical groups for various kinds of critical points.}

\section{Conditions and hypotheses}\label{h}

\begin{definition}
We say that $\varphi \in C^1(X,\mathbb{R})$ satisfies the {\it C-condition}, if it has the following property: 
Every sequence $\{u_n\}_{n \in \mathbb{N}} \subseteq X$ such that 
\begin{itemize}
\item $\{\varphi(u_n)\}_{n \in \mathbb{N}} \subseteq \mathbb{R}$ is bounded; and 
\item $(1 + \|u_n\|_X) \varphi'(u_n) \to 0 \mbox{ in }X^\ast \mbox{  as } n \to \infty$ {\color{black} ($X^\ast$ denotes the dual of $X$),}
\end{itemize}
admits a strongly convergent subsequence {\color{black} (see \cite{Ref11}, p. 366).} 
\end{definition}

Our hypotheses on the data of problem \eqref{eq0} will be the following:

\medskip
\noindent $H_0$: $p \in C^{0,1}(\overline{\Omega})$,  $\tau \in C(\overline{\Omega})$ and $1<\tau(z)<q<p(z) <N$ for all $z \in \overline{\Omega}$. 
\\
\noindent $H_1$: $f:\Omega \times \mathbb{R} \to \mathbb{R}$ is a Carath\'eodory function such that $f(z,0)=0$ for a.a. $z \in \Omega$ and 
\begin{itemize}
	\item[$(i)$] $|f (z,x)| \leq a(z) [1+|x|^{r(z)-1}]$ for a.a. $z \in \Omega$, all $x \in \mathbb{R}$, with $a \in L^\infty(\Omega)$, $ r \in C(\overline{\Omega})$ with $p(z)< r(z) < p_-^\ast$ for all $z \in \overline{\Omega}$;
	\item[$(ii)$] if $F(z,x)=\int_0^x f(z,s)ds$, then
	$\lim\limits_{x \to \pm \infty} \frac{F(z,x)}{|x|^{p_+}}=+\infty$ uniformly for a.a. $z \in \Omega$;
		\item[$(iii)$] 
	there exist $\mu \in C(\overline{\Omega})$ with
	$\mu(z)\in \left( (r_+-p_-)\frac{N}{p_-},p^\ast_-\right)$ for all $z \in \overline{\Omega}$, $\tau_+<\mu_-$ and {\color{black} a constant $\beta_0>0$ such that}
		$$\beta_0 \leq \liminf_{x \to \pm \infty} \dfrac{f(z,x)x-p_+F(z,x)}{|x|^{\mu(z)}}\mbox{ uniformly for a.a. $z \in \Omega$;}$$
	\item[$(iv)$] there exist $\eta \in L^\infty(\Omega)$ and $\widehat{\eta}>0$ such that
	\begin{align*} & \widehat{\lambda}_1(q)\leq \eta(z)\mbox{ for a.a. $z \in \Omega$, $\eta \not \equiv \widehat{\lambda}_1(q)$,}\\
		& \eta(z) \leq \liminf_{x \to 0} \dfrac{q F(z,x)}{|x|^{q}}  \leq \limsup_{x \to 0} \dfrac{q F(z,x)}{|x|^{q}}\leq \widehat{\eta}\mbox{ uniformly for a.a. $z \in \Omega$,}
	\end{align*}	
	(by $\widehat{\lambda}_1(q)$ we denote the principal eigenvalue of $(-\Delta_q, W_0^{1,q}(\Omega))$; we know $\widehat{\lambda}_1(q)>0$, see \cite{Ref6}, {\color{black} p. 741}).
\end{itemize}

\begin{remark}
Hypotheses $H_1\,(ii),(iii)$ imply that for a.a. $z \in \Omega$, $f(z,\cdot)$ is $(p_+-1)$-superlinear. We do not employ the AR-condition and this way we incorporate in our framework superlinear nonlinearities with ``slower'' growth as $x \to \pm \infty$. The following function satisfies hypothesis $H_1$ but it fails to satisfy the AR-condition:
$$f(z,x)=\begin{cases} \eta[ |x|^{q-2}x - |x|^{\theta(z)-2}x] & \mbox{if }|x|\leq1,\\
|x|^{p_+-2}x \ln |x|  & \mbox{if }1<|x|,
\end{cases}$$
with $\theta \in C(\overline{\Omega})$ and $ q < \theta(z)$ for all $z \in \overline{\Omega}$. 
\end{remark}

For $\lambda>0$, let $\varphi_\lambda : W_0^{1,p(z)}(\Omega) \to \mathbb{R}$ be the energy functional for problem \eqref{eq0} defined by 
$$ \varphi_\lambda(u)= \int_\Omega \frac{1}{p(z)}|\nabla u(z)|^{p(z)}dz + \frac{1}{q}\|\nabla u\|_q^q + \int_\Omega \frac{\lambda}{\tau(z)}|  u(z)|^{\tau(z)}dz -  \int_\Omega F(z,u)dz$$
for all $u \in W^{1,p(z)}_0(\Omega)$. Evidently, $\varphi_\lambda \in C^1(W_0^{1,p(z)}(\Omega))$. 

We also introduce the positive and negative truncations of $\varphi_\lambda(\cdot)$, namely the $C^1$-functionals $\varphi_\lambda^\pm : W_0^{1,p(z)}(\Omega) \to \mathbb{R}$ defined by
$$ \varphi_\lambda^\pm(u)= \int_\Omega \frac{1}{p(z)}|\nabla u(z)|^{p(z)}dz + \frac{1}{q}\|\nabla u\|_q^q + \int_\Omega \frac{\lambda}{\tau(z)}( u^\pm(z))^{\tau(z)}dz -  \int_\Omega F(z,\pm u^\pm)dz$$
for all $u \in W^{1,p(z)}_0(\Omega)$.  {\color{black} Recall $u^+=\max\{u,0\}$, $u^-=\max\{-u,0\}$.}\\

We can show that the functionals $ \varphi_\lambda^\pm(\cdot)$ and $ \varphi_\lambda(\cdot)$ satisfy the $C$-condition.

\begin{proposition}
	\label{prop3} If hypotheses $H_0$ and  $H_1$ hold and $\lambda>0$, then the functionals $ \varphi_\lambda^\pm(\cdot)$ and $ \varphi_\lambda(\cdot)$ satisfy the $C$-condition.
\end{proposition}
\begin{proof} We shall present the proof for the functional $ \varphi_\lambda^+(\cdot)$, the proofs for $ \varphi_\lambda^-(\cdot)$ and $ \varphi_\lambda(\cdot)$ are similar. So,  consider a sequence $\{u_n\}_{n \in \mathbb{N}}\subseteq W^{1,p(z)}_0(\Omega)$ such that
	\begin{align}&|\varphi_\lambda^+(u_n) \leq c_1 | \mbox{ for some $c_1>0$ and all $n \in \mathbb{N}$,} \label{eq1}\\ & (1 + \|u_n\|) (\varphi_\lambda^+)^\prime(u_n) \to 0 \mbox{ in }W^{-1,p^\prime(z)}(\Omega)\mbox{  as } n \to \infty. \label{eq2}\end{align}
	
{\color{black} Referring to \eqref{eq-Az},} by \eqref{eq2} we have
\begin{equation}
\label{eq3} \left| \langle A_{p(z)}(u_n),h\rangle +\langle A_{q}(u_n),h\rangle + \int_\Omega \lambda (u_n^+)^{\tau(z)-1}h dz -\int_\Omega f(z,u_n^+)h dz \right| \leq \frac{\varepsilon_n\|h\|}{1+\|u_n\|}
\end{equation}	
for all $h \in W^{1,p(z)}_0(\Omega)$, with $\varepsilon_n \to 0^+$. 

In \eqref{eq3} we choose $h=-u_n^- \in W^{1,p(z)}_0(\Omega)$ and obtain
\begin{align}
& \rho_p(\nabla u_n^-)\leq \varepsilon_n \mbox{ for all $n \in \mathbb{N}$,}\nonumber \\ \Rightarrow \quad & u_n^- \to 0 \mbox{ in $W^{1,p(z)}_0(\Omega)$ as $n \to \infty$ (see Proposition \ref{prop1}).}  \label{eq4}
\end{align}
	
From \eqref{eq1} and \eqref{eq4} we have
\begin{equation}
\label{eq5} \rho_p(\nabla u_n^+)+\frac{p_+}{q}\|\nabla u_n^+\|_q^q
+ \int_\Omega \frac{\lambda p_+}{\tau(z)}(u_n^+)^{\tau(z)}dz- \int_\Omega p_+F(z,u_n^+)dz \leq c_2 \end{equation}	
for some $c_2>0$
and
all $n \in \mathbb{N}$.

Also, if in \eqref{eq3} we use the test function $h=u_n^+ \in W^{1,p(z)}_0(\Omega)$, we obtain
\begin{equation}
\label{eq6} -\rho_p(\nabla u_n^+)-\|\nabla u_n^+\|_q^q
- \int_\Omega \lambda (u_n^+)^{\tau(z)}dz+ \int_\Omega f(z,u_n^+)u_n^+dz \leq \varepsilon_n \mbox{ for  all $n \in \mathbb{N}$.}\end{equation}	

We add \eqref{eq5} and \eqref{eq6} and obtain
$$\int_\Omega [f(z,u_n^+)u_n^+-p_+F(z,u_n^+)]dz \leq c_3 \mbox{ 	
for some $c_3>0$
and
 all $n \in \mathbb{N}$.}$$

From hypothesis $H_1\,(iii)$ we see that we can always assume that $\mu_-<r_-$. Hypotheses  $H_1\,(i),(iii)$ imply that there exist $\widehat{\beta}_0 \in (0,\beta_0)$ and $c_4>0$ such that 
\begin{equation}
	\label{eq8} \widehat{\beta}_0|x|^{\mu_-}-c_4 \leq f(z,x)x-p_+F(z,x) \mbox{ for a.a. $z \in \Omega$ and
	 all $x \in \mathbb{R}$.}
\end{equation}

We use \eqref{eq8} in \eqref{eq6} and obtain 
\begin{align}
	\nonumber & \|u_n^+\|^{\mu_-}_{\mu_-}\leq c_5 \mbox{ for some $c_5>0$
	and
	 all $n \in \mathbb{N}$,}\\ \label{eq9} \Rightarrow \quad & \{u_n^+\}_{n \in \mathbb{N}} \subseteq L^{\mu_-}(\Omega) \mbox{ is bounded.}
	\end{align}

Recall that $\mu_- < r_-\leq r_+<p_-^\ast$. So, we can find $t \in (0,1)$ such that 
\begin{equation}
	\label{eq10} \frac{1}{r_+}=\frac{1-t}{\mu_-}+\frac{t}{p_-^\ast}.
\end{equation}

Using the interpolation inequality (see Papageorgiou-Winkert \cite{Ref18}, p. 116), we have
\begin{align}
	\nonumber & \|u_n^+\|_{r_+} \leq \| u_n\|_{\mu_-}^{1-t}\|u_n\|_{p_-^\ast}^t \mbox{ for all $n \in \mathbb{N}$,}\\ \label{eq11}  \Rightarrow \quad & \|u_n^+\|_{r_+}^{r_+} \leq c_6 \|u_n^+\|^{tr_+} \mbox{ for some $c_6>0$, all $n \in \mathbb{N}$ (see \eqref{eq9}).}
	\end{align}

Also, from \eqref{eq3} with $h = u_n^+ \in W_0^{1,p(z)}(\Omega)$, we have
$$\rho_p(\nabla u_n^+) \leq c_7 + \int_\Omega f(z,u_n^+)u_n^+ dz \mbox{ for some $c_7>0$
and
 all $n \in \mathbb{N}$.}$$

Without loss of generality, we may assume that $\|u_n^+\| \geq 1$. Using hypothesis $H_1(i)$ and Proposition \ref{prop1}, we have
\begin{align}
	\nonumber  \|u_n^+\|^{p_-} &\leq c_8 [1+\|u_n^+\|_{r_+}^{r_+}] \mbox{ for some $c_8>0$,}\\
	& \leq c_9 [1+\|u_n^+\|^{tr_+}] \mbox{ for some $c_9>0$
	and
	 all $n \in \mathbb{N}$ (see \eqref{eq11}).} \label{eq12}
	\end{align}

From \eqref{eq10} we have
\begin{align*}
&	tr_+ = \frac{p_-^\ast(r_+-\mu_-)}{p_-^\ast-\mu_-}<p_- \mbox{ (see hypothesis $H_1\,(iii)$), }\\ \Rightarrow \quad & \{u_n^+\}_{n\in \mathbb{N}}\subseteq W_0^{1,p(z)}(\Omega) \mbox{ is bounded (see \eqref{eq12}),}\\ \Rightarrow \quad & \{u_n\}_{n \in \mathbb{N}}\subseteq W_0^{1,p(z)}(\Omega) \mbox{ is bounded (see \eqref{eq4}).}
	\end{align*}

So, we may assume that
\begin{equation}
	\label{eq13}u_n \xrightarrow{w} u \mbox{  in  $W_0^{1,p(z)}(\Omega)$ and $u_n \to u$ in $L^{r(z)}(\Omega)$. }
\end{equation}

In \eqref{eq3} we choose $h=u_n-u \in W_0^{1,p(z)}(\Omega)$, pass to the limit as $n \to \infty$, and use \eqref{eq13}. We obtain 
\begin{align*}
	\nonumber & \lim_{n \to \infty}\left[ \langle A_{p(z)}(u_n),u_n-u\rangle +\langle A_{q}(u_n),u_n-u\rangle \right]=0,\\
	\Rightarrow \quad  & \nonumber \limsup_{n \to \infty}\left[ \langle A_{p(z)}(u_n),u_n-u\rangle +\langle A_{q}(u),u_n-u\rangle \right]\leq 0,\\
	\nonumber & \hskip 4cm \mbox{(since $A_q(\cdot)$ is monotone),}\\
	\Rightarrow \quad  & \nonumber \limsup_{n \to \infty}  \langle A_{p(z)}(u_n),u_n-u\rangle \leq 0 \quad \mbox{(see \eqref{eq13}),}\\ \Rightarrow \quad  & u_n \to u \quad \mbox{in $W_0^{1,p(z)}(\Omega)$ (see Proposition \ref{prop2}).}
\end{align*}

This proves that the functional $\varphi_\lambda^+(\cdot)$ satisfies the $C$-condition. In a similar fashion we show that 
$\varphi_\lambda^-(\cdot)$ and $\varphi_\lambda(\cdot)$ also satisfy the $C$-condition.
	\end{proof}

\section{Auxiliary propositions}\label{ap}

We shall prove two propositions needed for the proof of the main result.
\begin{proposition}
	\label{prop4} If hypotheses $H_0$
	and
	 $H_1$ hold and $\lambda>0$, then there exist $\rho_0,c_0>0$ such that $\varphi_\lambda^\pm(u)\geq c_0>0$ for all $u \in W_0^{1,p(z)}(\Omega)$, $\|u\|=\rho_0$.
\end{proposition}

\begin{proof}
	On account of hypothesis $H_1 \, (iv)$, we have 
	\begin{equation}
	\label{eq14}\lim_{x \to 0^+} \frac{F(z,x)}{x^{\tau(z)}}= \lim_{x \to 0^+} \left[\frac{F(z,x)}{x^{q}}x^{q-\tau(z)}\right]= 0 \mbox{ (recall that $\tau_+<q$).}
	\end{equation}
	
Then \eqref{eq14} and hypothesis $H_1 \, (i)$ imply that given $\varepsilon>0$, we can find $c_{10}=c_{10}(\varepsilon)>0$ such that
\begin{equation*}
\label{eq15}F_+(z,x)\leq \frac{\varepsilon}{\tau_+}|x|^{\tau(z)}+c_{10}|x|^{r_-} \mbox{ for a.a. $z \in \Omega$ and all $x \in \mathbb{R}$.}
\end{equation*}	
	
	For $u \in W_0^{1,p(z)}(\Omega)$ with $\|u\|\leq 1$, we have 
	$$ \varphi_\lambda^+(u)\geq \frac{1}{p_+}\rho_p(\nabla u)+ \frac{1}{\tau_+}[\lambda -\varepsilon]\rho_\tau(u)-c_{11}\|u\|^{r_-} \mbox{ for some $c_{11}>0$ (since $\|u\|_{\tau(z)}\leq 1$).}$$	
	
Choosing $\varepsilon \in (0,\lambda)$ and recalling that $\|u\|\leq 1$, we have
	$$ \varphi_\lambda^+(u)\geq \frac{1}{p_+} \|u\|^{p_+}-c_{11}\|u\|^{r_-} \mbox{ (see Proposition \ref{prop1}).}$$
	
	Recall that $p_+< r_-$. So, by choosing $\rho_0 \in (0,1)$ sufficiently small, we obtain
		$$ \varphi_\lambda^+(u)\geq c_0>0 \mbox{ for all $u \in W_0^{1,p(z)}(\Omega)$, $\|u\|=\rho_0$.}$$
	
	Similarly for $\varphi_\lambda^-(\cdot)$.	
	\end{proof}

Recall that $\widehat{\lambda}_1(q)>0$ is the principal eigenvalue of $(-\Delta_q,W_0^{1,q}(\Omega))$. Also, by $\widehat{u}_1=\widehat{u}_1(q)$ we denote the corresponding positive $L^q$-normalized (that is, $\|\widehat{u}_1\|_q=1$) eigenfunction. We know that $\widehat{u}_1 \in C_0^1(\overline{\Omega})$ and $\widehat{u}_1(z)>0$ for all $z \in \Omega$ (see \cite{Ref6},  {\color{black} Theorem 6.2.9, p. 739}).

\begin{proposition}
	\label{prop5} If hypotheses $H_0$
	and
	$H_1$ hold, then there exist $\lambda^\ast>0$ and $t_\pm>0$ such that $\varphi_\lambda^\pm(\pm t_\pm  {\color{black}\widehat{u}_1})<0$ for all $\lambda \in (0,\lambda^\ast)$.
\end{proposition}

\begin{proof}
		On account of hypotheses $H_1 \,(i), (iv)$, given $\varepsilon>0$, we can find $c_{12}=c_{12}(\varepsilon)>0$ such that 
		\begin{equation*}
		\label{eq16}F_+(z,x)\geq \frac{1}{q}[\eta(z)-\varepsilon]|x|^{q}-c_{12}|x|^{r_-} \mbox{ for a.a. $z \in \Omega$ and all $x \geq 0$.}
		\end{equation*}	
		
Then for $t \in (0,1]$ we have 
		$$ \varphi_\lambda^+(t \widehat{u}_1)\leq \frac{t^{p_-}}{p_-}\rho_p(\nabla \widehat{u}_1)+ \frac{t^q}{q}\left[ \int_\Omega (\widehat{\lambda}_1(q)-\eta(z))\widehat{u}_1^q dz+\varepsilon \right] + \frac{\lambda t^{\tau_-}}{\tau_-}\rho_\tau(\widehat{u}_1)+c_{12}t^{r_-}\|\widehat{u}_1\|^{r_-}_{r_-} .$$
		
As we have mentioned earlier, $\widehat{u}_1(z)>0$ for all $z \in \Omega$. This fact, combined with hypothesis $H_1\,(iv)$, implies that 
$$\widehat{\mu}=\int_\Omega (\eta(z)-\widehat{\lambda}_1(q))\widehat{u}_1^qdz>0.$$

So, choosing $\varepsilon \in (0,\widehat{\mu})$, we obtain
\begin{align}
\nonumber \varphi_\lambda(t \widehat{u}_1) & \leq c_{13}[t^{p_-}+\lambda t^{\tau_-}]-c_{14}t^q \mbox{ for some $c_{13},c_{14}>0$}\\ \label{eq17} & = [c_{13}(t^{p_--q}+\lambda t^{\tau_--q})-c_{14}]t^q.
\end{align}

Consider the function
$$\xi_\lambda(t)=t^{p_--q}+\lambda t^{\tau_--q} \mbox{ for $t>0$.}$$
		
Since $\tau_-<q<p_-$, we see that
$$\xi_\lambda(t)\to +\infty \mbox{ as $t \to 0^+$ and as $t \to +\infty$.}$$

Therefore there exists $t_+>0$ such that
		\begin{align}
\nonumber		&  \xi_\lambda(t_+)=\inf \{\xi_\lambda(t):t>0\},\\
	\nonumber 	\Rightarrow \quad & \xi^\prime_\lambda(t_+)=0,\\ \nonumber \Rightarrow \quad & (p_-q)t_+^{p_- - \tau_-}=\lambda (q-\tau_-),\\ \label{eq18} \Rightarrow \quad & t_+= \left[\frac{\lambda(q-\tau_-)}{p_--q}\right]^{\frac{1}{p_--\tau_-}}.
		\end{align}
		
Using \eqref{eq18}, we see that
$$\xi_\lambda (t_+)\to 0^+ \mbox{ as $\lambda \to 0^+$.}$$

Hence we can find $\lambda_1^\ast>0$ such that
\begin{align*}
& \xi_\lambda(t_+)<\frac{c_{14}}{c_{13}} \mbox{ for all $\lambda \in (0,\lambda_1^\ast)$,}\\ \Rightarrow \quad & \varphi_\lambda^+(t_+\widehat{u}_1)<0 \mbox{ for all $\lambda \in (0,\lambda_1^\ast)$ (see \eqref{eq17}).}
\end{align*}

Similarly working with $\varphi_\lambda^-(\cdot)$, we produce $\lambda_2^\ast>0$ and $t_->0$ such that
$$\varphi_\lambda^-(-t_-\widehat{u}_1)<0 \mbox{ for all $\lambda \in (0,\lambda_2^\ast)$.}$$

Finally let $\lambda^\ast =\min\{\lambda_1^\ast,\lambda_2^\ast\}$.
\end{proof}	

\begin{remark}
We can always choose $\lambda^\ast>0$ small so that
\begin{equation}
\label{eq19}t_\pm=t_\pm (\lambda)\in (0,\rho_0) \mbox{ for all $\lambda \in (0,\lambda^\ast)$ ($\rho_0>0$ is as in  {\color{black} Proposition \ref{prop4}}).}
\end{equation}
\end{remark}

\section{Proof of the main theorem}\label{sec3}
We shall break down the proof of Theorem \ref{th9} into two steps (5.1 and 5.2).

\subsection{Existence of two solutions}
First, we shall produce two nontrivial constant sign solutions. In what follows, we shall denote $C_+=\{u\in C_0^1(\overline{\Omega}): 0 \leq u(z) \mbox{ for all $z \in \overline{\Omega}$} \}$.

\begin{proposition}
	\label{prop6} If hypotheses $H_0$
	and $H_1$ hold and $\lambda \in (0,\lambda^\ast)$, then problem \eqref{eq0} has at least two constant sign solutions $u_0 \in C_+ \setminus \{0\}$, $v_0 \in (-C_+) \setminus \{0\}$ and both are local minimizers of the energy functional $\varphi_\lambda(\cdot)$.
\end{proposition}

\begin{proof}
We introduce the closed ball
$$\overline{B}_{\rho_0}=\{u \in W_0^{1,p(z)}(\Omega):\|u\|\leq \rho_0 \}$$
 {\color{black} with $\rho_0>0$ as in Proposition \ref{prop4} } and consider the minimization problem
\begin{equation}
\label{eq20} \inf \{\varphi_\lambda^+(u):u \in  \overline{B}_{\rho_0} \}=m_\lambda^+.
\end{equation}

The anisotropic Sobolev embedding theorem (see Section \ref{sec:2}), implies that $\varphi_\lambda^+(\cdot)$ is sequentially weakly lower semicontinuous. Also the reflexivity of $W_0^{1,p(z)}(\Omega)$ and the Eberlein-Smulian theorem  {\color{black} (see \cite{Ref18}, p. 221)} imply that $\overline{B}_{\rho_0}$ is sequentially weakly compact. So, by the Weierstrass-Tonelli theorem  {\color{black} (see \cite{Ref18}, p. 78)}, we can find $u_0 \in \overline{B}_{\rho_0}$ such that
\begin{align}
\label{eq21} &\varphi_\lambda^+(u_0)=m_\lambda^+ \leq \varphi_\lambda^+(t_+\widehat{u}_1)<0=\varphi_\lambda^+(0) \mbox{ (see \eqref{eq19}, \eqref{eq20} and Proposition \ref{prop5}),}\\ \nonumber \Rightarrow \quad & u_0 \neq 0.
\end{align}

From \eqref{eq21} and Proposition \ref{prop4}, we have
$$0<\|u_0\|<\rho_0.$$

Hence we have 
\begin{align}
\nonumber & (\varphi_\lambda^+)^\prime (u_0)=0,\\ \Rightarrow \quad & \langle A_{p(z)}(u_0),h \rangle +  \langle A_{q}(u_0),h \rangle = \int_\Omega f(z,u_0^+)h dz -\lambda \int_\Omega (u_0^+)^{\tau(z)-1} h dz \label{eq22}
\end{align}
for all $h \in W_0^{1,p(z)}(\Omega)$. In \eqref{eq22} we choose $h =-u_0^-\in W_0^{1,p(z)}(\Omega)$ and obtain 
\begin{align*}
& \rho_p(\nabla u_0^-)+\|\nabla u_0^-\|_q^q =0,\\ \Rightarrow \quad & u_0\geq 0, \, u_0 \neq 0.
\end{align*}

By Papageorgiou-R\u{a}dulescu-Zhang \cite[Proposition A.1]{Ref13}, we know that $u_0 \in L^\infty(\Omega)$. Then the anisotropic regularity theory (see Fan {\color{black}\cite[Theorem 1.3]{Ref4} } and Tan-Fang {\color{black}\cite[Corollary 3.1]{Ref20})} implies $u_0 \in C_+ \setminus \{0\}$. So, we have produced a positive smooth solution of \eqref{eq0} for $\lambda \in (0,\lambda^\ast)$. Similarly working with functional $\varphi_\lambda^-(\cdot)$, we produce a negative solution $v_0$ of \eqref{eq0} ($\lambda \in (0,\lambda^\ast)$) such that 
$$v_0 \in (-C_+) \setminus \{0\}.$$

Finally, we show that $u_0$ and $v_0$ are both local minimizers of the energy functional $\varphi_\lambda(\cdot)$. We shall present  the proof for $u_0$, the proof for $v_0$ is similar. From the first part of the proof, we know that $u_0$ is a local $C_0^1(\overline{\Omega})$-minimizer of $\varphi_\lambda^+(\cdot)$. So, we can find $\rho_1>0$ such that
\begin{equation}
\label{eq23} \varphi_\lambda^+(u_0)\leq \varphi_\lambda^+(u) \mbox{ for all $u \in \overline{B}_{\rho_1}^{C_0^1}(u_0)=\{u \in C_0^1(\overline{\Omega}): \|u-u_0\|_{C_0^1(\overline{\Omega})}\leq \rho_1 \}$}.
\end{equation}

For $u \in \overline{B}_{\rho_1}^{C_0^1}(u_0)$ we have
\begin{align}
\nonumber & \varphi_\lambda(u)-\varphi_\lambda(u_0)\\
\nonumber = & \, \varphi_\lambda(u)-\varphi_\lambda^+(u_0) \mbox{ (since $\varphi_\lambda \big|_{C_+}=\varphi_\lambda^+ \big|_{C_+} $)}\\
\nonumber \geq & \, \varphi_\lambda(u)-\varphi_\lambda^+(u) \mbox{ (see \eqref{eq23})}\\
\nonumber \geq & \frac{\lambda}{\tau_+} \int_\Omega [|u|^{\tau(z)}-(u^+)^{\tau(z)} ]dz - \int_\Omega [F(z,u)-F(z,u^+)]dz\\ \label{eq24} = & \, \frac{\lambda}{\tau_+} \rho_\tau (u^-)- \int_\Omega F(z,-u^-)dz.
\end{align}

	On account of hypotheses $H_1 \,(i), (iv)$ we can find $c_{15}>0$ such that 
\begin{equation}
\label{eq25}F(z,x)\leq c_{15}[|x|^{q}+|x|^{r_+}] \mbox{ for a.a. $z \in \Omega$ and all $x \in \mathbb{R}$.}
\end{equation}	

Using \eqref{eq25} in \eqref{eq24}, we obtain
\begin{align}
\nonumber \varphi_\lambda(u)-\varphi_\lambda(u_0) & \geq  \frac{\lambda}{\tau_+}\rho_\tau(u^-)-c_{15}\int_\Omega [(u^-)^q+(u^-)^{r_+}]dz\\
\label{eq26} & \geq \frac{\lambda}{\tau_+}\rho_\tau(u^-)-c_{15}\int_\Omega [\|u^-\|^{q-\tau(z)}_\infty+\|u^-\|^{r_+-\tau(z)}_\infty] (u^-)^{\tau(z)}dz.
\end{align}

Recall that $u_0 \in C_+ \setminus \{0\}$ and $u \in \overline{B}_{\rho_1}^{C_0^1}(u_0)$. So, by choosing $\rho_1>0$ even smaller if necessary, we can have that $\|u^-\|_\infty \leq 1$. Hence
\begin{equation}
\label{eq27} \|u^-\|^{q-\tau(z)}_\infty \leq \|u^-\|^{q-\tau_+}_\infty, \, \|u^-\|^{r_+-\tau(z)}_\infty \leq \|u^-\|^{r_+-\tau_+}_\infty.
\end{equation}

We return to \eqref{eq26} and use \eqref{eq27}. We obtain
$$\varphi_\lambda(u)-\varphi_\lambda(u_0)  \geq \left[\frac{\lambda}{\tau_+}  -c_{15}(\|u^-\|^{q-\tau_+}_\infty+ \|u^-\|^{r_+-\tau_+}_\infty ) \right]\rho_\tau(u^-).
$$

Note that $\|u^-\|_\infty \to 0^+$ as $\rho_1 \to 0^+$. Therefore we can choose $\rho_1>0$ so small that
$$\varphi_\lambda(u) \geq \varphi_\lambda(u_0) \mbox{ for all $u \in \overline{B}_{\rho_1}^{C_0^1}(u_0)$.}$$

This means that $u_0$ is a local $C_0^1(\overline{\Omega})$-minimizer of $\varphi_\lambda(\cdot)$. Then Proposition A.3 of Papageorgiou-R\u{a}dulescu-Zhang \cite{Ref13}, implies that $u_0$ is a local $W_0^{1,p(z)}(\Omega)$-minimizer of $\varphi_\lambda(\cdot)$. Similarly we show that $v_0 \in (-C_+)\setminus \{0\}$ is a local minimizer of the energy functional $\varphi_\lambda(\cdot)$.
\end{proof}

\begin{proposition}
	\label{prop7} If hypotheses $H_0$
	and
	 $H_1$ hold and $\lambda >0$, then $u=0$ is a local minimizer of the energy functional $\varphi_\lambda(\cdot)$.
\end{proposition}

\begin{proof}
	Let $u \in C_0^1(\overline{\Omega})$ with  $\|u\|_{C_0^1(\overline{\Omega})}\leq 1$. We have
	\begin{align*}
	\varphi_\lambda(u)-\varphi_\lambda(0) & = \varphi_\lambda (u)\\
	& \geq \frac{\lambda}{\tau_+}\rho_\tau(u)-\int_\Omega F(z,u)dz \\
	& \geq \left[  \frac{\lambda}{\tau_+}-c_{15}(\|u\|^{q-\tau_+}_\infty+ \|u\|^{r_+-\tau_+}_\infty ) \right]\rho_\tau(u) \mbox{ (see \eqref{eq25}).}
	\end{align*}
	
	Choosing $\rho>0$ small enough, we see that 
	\begin{align*}
	& \varphi_\lambda(u) \geq 0=\varphi_\lambda(0) \mbox{ for all $u \in \overline{B}_{\rho}^{C_0^1}(0)$,}
	\\ \Rightarrow \quad  & u=0 \mbox{ is a local $C_0^1(\overline{\Omega})$-minimizer of  $\varphi_\lambda(\cdot)$,}
		\\ \Rightarrow \quad  & u=0 \mbox{ is a local $W_0^{1,p(z)}(\Omega)$-minimizer of  $\varphi_\lambda(\cdot)$ (see \cite{Ref13}).}
	\end{align*}
\end{proof}
\subsection{Existence of the third solution}
Now we are ready to produce the third nontrivial solution for problem \eqref{eq0}, $\lambda \in (0,\lambda^\ast)$.

	\begin{proposition}
		\label{prop8} If hypotheses $H_0$
		and  $H_1$ hold and $\lambda \in (0,\lambda^\ast)$, then problem \eqref{eq0} has the third solution $y_0 \in C_0^1(\overline{\Omega})$ and $y_0 \not \in \{0,u_0,v_0\}$.
	\end{proposition}
	
	\begin{proof}
	From the anisotropic regularity theory {\color{black} (see \cite{Ref4}, \cite{Ref20})}, we have that $K_{\varphi_\lambda} \subseteq C_0^1(\overline{\Omega})$. Since the critical points of $\varphi_\lambda(\cdot)$ are the weak solutions of \eqref{eq0}, we may assume that $K_{\varphi_\lambda}$ is finite or otherwise we would already have an infinity of nontrivial smooth solutions for \eqref{eq0} and so we would be done. Then Proposition \ref{prop7} and   \cite[Theorem 5.7.6, p. 449]{Ref11}, imply that we can find $\widehat{\rho}>0$ such that 
	\begin{equation}
	\label{eq28} \varphi_\lambda(0)=0< \inf \{\varphi_\lambda(u): \|u\|=\widehat{\rho}\}=\widehat{m}_\lambda.
	\end{equation}
	
Also, if $u \in C_+$ with $u(z)>0$ for all $z \in \Omega$, then on account of hypothesis $H_1\,(ii)$, we have 
\begin{equation}
\label{eq29} \varphi_\lambda(tu) \to -\infty \mbox{ as $t \to +\infty$.}
\end{equation}

Then \eqref{eq28}, \eqref{eq29} and Proposition \ref{prop3}, permit the use of the Mountain Pass Theorem {\color{black} (see \cite{Ref11}, p. 401)}. So, we can find $y_0 \in W_0^{1,p(z)}(\Omega)$ such that 
\begin{align*}
& y_0 \in K_{\varphi_\lambda}, \, \varphi_\lambda(0)=0 < \widehat{m}_\lambda \leq \varphi_\lambda(y_0),\\ \Rightarrow \quad & y_0 \neq 0.
\end{align*}

Moreover,  \cite[Corollary 6.6.9, p. 533]{Ref11} implies that 
\begin{equation}
\label{eq30} C_1(\varphi_\lambda,y_0)\neq 0.
\end{equation}

On the other hand from Proposition \ref{prop6}, we infer that 
\begin{equation}
\label{eq31} C_k(\varphi_\lambda,u_0)= C_k(\varphi_\lambda,v_0)=\delta_{k,0}\mathbb{Z} \mbox{ for all $k \in \mathbb{N}_0$.}
\end{equation}

Comparing \eqref{eq30} and \eqref{eq31}, we conclude that 
$$ y_0 \neq u_0, \, y_0 \neq v_0.$$

The anisotropic regularity theory implies that $y_0 \in C_0^1(\overline{\Omega})$.
	\end{proof}

This also completes the proof of Theorem \ref{th9}.\qed

\section*{Acknowledgements} The authors thank the referee for his/her remarks. Repov\v{s} was supported by the Slovenian Research Agency grants P1-0292, J1-4031, J1-4001, N1-0278, and N1-0114.

\end{document}